\documentclass[11pt]{article}
\usepackage{amsmath,amssymb,amsthm,mathtools,diagrams,theoremref,a4wide}

\DeclareMathOperator{\Hom}{Hom}
\DeclareMathOperator{\tr}{tr}
\DeclareMathOperator{\id}{id}

\DeclareMathOperator{\End}{End}

\DeclareMathOperator{\Sym}{Sym}

\DeclareMathOperator{\Aut}{Aut}

\DeclareMathOperator{\Res}{Res}
\DeclareMathOperator{\Ind}{Ind}

\DeclareMathOperator{\Mod}{-Mod}

\DeclareMathOperator{\Irr}{Irr}

\DeclareMathOperator{\FI}{FI}
\DeclareMathOperator{\Rat}{Rat}

\DeclareMathOperator{\Conf}{Conf}
\DeclareMathOperator{\Frob}{Frob}
\DeclareMathOperator{\UConf}{UConf}

\DeclareMathOperator{\Gal}{Gal}
\DeclareMathOperator{\stab}{stab}

\newtheorem{thm}{Theorem}[section]
\newtheorem{prop}[thm]{Proposition}

\newtheorem{cor}[thm]{Corollary}
\begin{document}

\title{Representation theory and arithmetic statistics\\ of spaces of 0-cycles}
\author{Kevin Casto}
\maketitle

\begin{abstract}
We continue the study of a general class of spaces of 0-cycles on a manifold defined and begun by Farb-Wolfson-Wood \cite{FWW}. Using work of Gadish \cite{Ga1} on linear subspace arrangements, we obtain representation stability for the cohomology of the ordered version of these spaces. We establish subexponential bounds on the growth of unstable cohomology, and the Grothendieck-Lefschetz trace formula then allows us to translate these topological stability phenomena to stabilization of statistics for spaces of 0-cycles over finite fields. In particular, we show that the average value of certain arithmetic quantities associated to rational maps over finite fields stabilizes as the degree goes to infinity.
\end{abstract}

\section{Introduction}
In \cite{FWW}, Farb-Wolfson-Wood studied a general class of spaces of 0-cycles on a manifold \(X\) that they call \(\mathcal{Z}^{(d_1, \dots d_m)}_n(X)\), which is the subspace of \(\prod_i \Sym^{d_i}(X)\) where no point of \(X\) appears \(n\) (or more) times in every \(\Sym^{d_i}(X)\). (Here \(\Sym^d(X) := X^d/S_d\)). In particular, \(\mathcal{Z}^{d}_2(X) = \UConf_d(X)\), the space of unordered configurations of \(d\) points in \(X\), and \(\mathcal{Z}^{(d, \dots, d)}_1(\mathbb{C}) = \Rat^*_d(\mathbb{CP}^{m-1})\), the space of degree \(d\), based rational maps \(\mathbb{CP}^1 \to \mathbb{CP}^{m-1}\) with \(f(\infty) = [1 : \cdots : 1]\). They analyze these spaces in detail, and among other results, prove that they satisfy homological stability.

In order to study the spaces \(\mathcal{Z}^{(d_1, \dots d_m)}_n(X)\), Farb-Wolfson-Wood pass to an ordered version that they denote \(\widetilde{\mathcal{Z}}^{(d_1, \dots d_m)}_n(X)\), which is the subspace of \(\prod_i X^{d_i}\) where no point of \(X\) appears \(n\) or more times in every \(X^{d_i}\). The space \(\widetilde{\mathcal{Z}}^{(d_1, \dots d_m)}_n(X)\) has an action of 
\[S_{\mathbf d} := S_{d_1} \times \cdots \times S_{d_m}\]
with quotient \(\mathcal{Z}^{(d_1, \dots d_m)}_n(X)\). They analyze the Leray spectral sequence for the inclusion \(\widetilde{\mathcal{Z}}^{\mathbf d}_n(X) \hookrightarrow X^{d_1 + \cdots + d_m}\), generalizing the Leray spectral sequence of the inclusion \(\Conf_n(X) \hookrightarrow X^n\) studied by Totaro \cite{To}. The key point is that locally, \(\widetilde{\mathcal{Z}}^{\mathbf d}_n\) looks like the complement of a linear subspace arrangement. The combinatorics of this subspace arrangement are substantially more complicated than in the specific casef of \(\Conf_n\), but they are able to use the theory of ``shellability'' of partial orders to understand it. By transfer, \(H^*(\mathcal{Z}^{\mathbf d}_n) = H^*(\widetilde{\mathcal{Z}}^{\mathbf d}_n)^{S_{\mathbf d}}\), so by understanding the invariants of the spectral sequence, they are able to draw conclusions about \(H^*(\mathcal{Z}^{\mathbf d}_n)\).

\subsection*{Representation stability}

In this paper, we extend \cite{FWW} by studying \(H^*(\mathcal{Z}^{\mathbf d}_n)\) as an \(S_{\mathbf d}\)-representation. To wit, we consider all such \(\mathbf d\) simultaneously, and think of \(H^*(\mathcal{Z}^{\mathbf d}_n)\) as an \emph{\(\FI^m\)-module}. Recall that FI is the category of finite sets and injections studied by Church-Ellenberg-Farb \cite{CEF}. \(\FI^m\) is just the product category, which was studied by Gadish \cite{Ga2}; see \S2 for a precise definition of \(\FI^m\)-modules. By analyzing the Leray spectral sequence mentioned, we prove that \(H^*(\mathcal{Z}^{\mathbf d}_n)\) is \emph{finitely generated} as an \(\FI^m\)-module, meaning that it has a finite subset that is not contained in any smaller sub-\(\FI^m\)-module.

Our work implies the following results about \(H^*(\widetilde{\mathcal{Z}}^{\mathbf d}_n)\). Recall that the irreducible representations of \(S_{\mathbf d}\) are parameterized by lists of partitions \(\mathbf \lambda = (\lambda^1, \dots, \lambda^m)\), where \(\lambda^i\) is a partition of \(d_i\). Denote the irreducible representation of \(S_{\mathbf d}\) associated to \(\mathbf \lambda\) by \(\Irr(\mathbf \lambda)\). For each \(i, j\), let \(X^i_j\) be the class function on \(\bigcup_{\mathbf d} S_{\mathbf d}\) that counts the number of \(j\)-cycles on \(S_{d_i}\). Define a \emph{character polynomial} to be an element of \(\mathbb{Q}[\{X^i_j\}]\).

\begin{thm}[\bfseries Polynomiality of characters and representation stability]
\thlabel{top_fg}
Fix \(n\) and \(i\), let \(X\) be a connected manifold of dimension at least 2, and let \(V_{\mathbf d} = H^i(\widetilde{\mathcal{Z}}^{\mathbf d}_n; \mathbb{Q})\). Then: \begin{enumerate}
\item There is a single character polynomial \(P \in \mathbb{Q}[\{X^i_j\}]\) such that the character \(\chi_{V_{\mathbf d}} = P\) for all \(d_i \gg 0\). In particular, there is a polynomial \(Q \in \mathbb{Q}[y_1, \dots, y_m]\) such that \(\dim V_{\mathbf d} = Q(d_1, \dots, d_m)\) for all \(d_i \gg 0\).
\item The multiplicity of each irreducible \(S_{\mathbf d}\)-representation in \(V_{\mathbf d}\) is independent of \(\mathbf d\) when all \(d_i \gg 0\).
\end{enumerate}
\end{thm}
Plugging  in \(\mathbf d = (d)\) and \(n = 2\) into \thref{top_fg} recovers representation stability for \(\Conf_d(X)\), as proven by Church \cite[Thm 1]{Chu} and Church-Ellenberg-Farb \cite[Thm 1.8]{CEF}. Furthermore, just looking at the multiplicity of the trivial representation in \thref{top_fg}.2 recovers \cite[Thm 1.6]{FWW}.

Furthermore, when \(X\) is a smooth complex variety, Farb-Wolfson-Wood obtain results about the mixed Hodge structure of \(\mathcal{Z}^{\mathbf d}_n(X)\). We extend their analysis by considering the case when \(X\) is a smooth scheme over \(\mathbb{Z}[1/N]\). Thus we can consider the complex points \(X(\mathbb{C})\), but also the finite field points \(X(\mathbb{F}_q)\). We then obtain results on \emph{\'etale representation stability} in the sense of \cite{FW2} (see \S4 below for the precise definitions).

\begin{thm}[\bfseries \'Etale representation stability]
\thlabel{etale_stab}
Let \(X\) be a smooth scheme over \(\mathbb{Z}[1/N]\) with geometrically connected fibers that is smoothly compactifiable. Let \(K\) be either a number field or a finite field over \(\mathbb{Z}[1/N]\). For each \(n\) and \(i\), the \emph{\(\Gal(\overline{K}/K)\)-\(\FI^m\)-module} \(H^i_{\acute et}(\widetilde{\mathcal{Z}}^{\bullet}_n(X); \mathbb{Q}_l)\) is finitely generated.
\end{thm}

\subsection*{Stability of arithmetic statistics}

We would like to use \thref{etale_stab} to obtain asymptotic statistics for the number of points of \(\mathcal{Z}^{\mathbf d}_n(X)(\mathbb{F}_q)\). To do so, we need the next result, which gives bounds on \(H^i(\mathcal{Z}^{\mathbf d}_n(X))\) as \(i\) varies.

\begin{thm}[\bfseries Convergence]
Let \(X\) be a connected manifold of dimension at least 2 with \(\dim H^*(X) < \infty\). Then for each character polynomial \(P\), the inner product \(|\langle P, H^i(\mathcal{Z}^{\mathbf d}_n(X))\rangle|\) is bounded subexponentially in \(i\) and uniformly in \(n\).
\end{thm}

Finally, we use the Grothendieck-Lefschetz trace formula, along with Theorems 1.2 and 1.3, to obtain the following results on arithmetic statistics.
\begin{thm}[\bfseries Stability of arithmetic statistics]
Let \(X\) be a smooth quasiprojective scheme over \(\mathbb{Z}[1/N]\) with geometrically connected fibers. Then for any \(n\) and any character polynomial \(P\),
\[ \lim_{\mathbf d \to \infty} q^{-|\mathbf d| \dim X} \sum_{y \in \mathcal{Z}^{\mathbf d}_n(X)(\mathbb{F}_q)} P(y) = \sum_{i=0}^\infty (-1)^i \tr\left(\Frob_q : \langle H^i(\widetilde{\mathcal{Z}}^{\mathbf d}_n(X)), P \rangle\right) \]
Namely, both the limit on the left and the series on the right converge, and they converge
to the same limit.
\end{thm}
In particular, Theorems 1.2-1.4 recover the results for \(\Conf_d(X)\) proven by Farb-Wolfson \cite[Thm C]{FW2}, when \(\mathbf d = (d)\) and \(n = 2\).

In the case where \(\mathbf d = (\overbrace{d,\dots,d}^{m})\), \(n = 1\), and \(X = \mathbb{A}^1\), where \(\mathcal{Z}^{(d, \dots, d)}_1(\mathbb{A}^1) = \Rat^*_d(\mathbb{P}^{m-1})\), we obtain the following result about rational maps over finite fields.

\begin{thm}
For any prime power \(q\) and any character polynomial \(P\), 
\[ \lim_{d \to \infty} q^{-m d} \sum_{f \in \Rat^*_d(\mathbb{P}^{m-1})(\mathbb{F}_q)} P(f) = \sum_{i=0}^\infty (-1)^i \tr\left(\Frob_q : \langle H^i(\widetilde{\Rat}^*_d(\mathbb{CP}^{m-1})), P \rangle\right) \]
Namely, both the limit on the left and the series on the right converge, and they converge
to the same limit.
\end{thm}

In general, the space \(\widetilde{\Rat}^*_d(\mathbb{CP}^{m-1})\) is the complement of a linear subspace arrangement, one that is not well-behaved enough for us to say what the eigenvalues of Frobenius acting on its \'etale cohomology are. However, in the case \(m = 2\) the space \(\widetilde{\Rat}^*_d(\mathbb{CP}^{1})\) is the complement of a hyperplane arrangement, and therefore by \cite{BE} the action of \(\Frob_q\) on \(H^i(\widetilde{\Rat}^*_d(\mathbb{CP}^{1}))\) is multiplication by \(q^{-i}\). So we obtain
\[ \lim_{d \to \infty} q^{-2 d} \sum_{f \in \Rat^*_d(\mathbb{P}^{1})(\mathbb{F}_q)} P(f) = \sum_{i=0}^\infty (-1)^i \langle H^i(\widetilde{\Rat}^*_d(\mathbb{CP}^{1})), P \rangle\, q^{-i} \]

As an example of Theorem 1.5, in the case \(P = 1\), we obtain
\[ \lim_{d \to \infty} q^{-m d} \left|\Rat^*_d(\mathbb{P}^m)(\mathbb{F}_q)\right| = \sum_{i=0}^\infty (-1)^i \tr\left(\Frob_q : H^i(\Rat^*_d(\mathbb{P}^{m-1})) \right)  \]
so that the number of such rational maps, as the degree goes to infinity, stabilizes to the series on the right. As another example, if \(P = X^i_1\), then \(X^i_1(f)\) counts the number (with multiplicity) of \(\mathbb{F}_q\)-rational intersection points of the image of \(f\) in \(\mathbb{P}^m\) with the hyperplane \(\{x_i = 0\}\). Thus,
\[ \lim_{d \to \infty} q^{-m d} \sum_{f \in \Rat^*_d(\mathbb{P}^m)(\mathbb{F}_q)} \#\{f^{-1}\{x_i = 0\}\} = \sum_{i=0}^\infty (-1)^i \tr\left(\Frob_q : \langle X^i_1, H^i(\widetilde{\Rat}^*_d(\mathbb{P}^{m-1}))\rangle \right)  \]
so that the \emph{average} number of intersection points, across all such rational maps, stabilizes as the degree goes to infinity to the series on the right.

\subsection*{Related work}
Most of what we prove in \S2 was independently proven by Li-Yu \cite{LY}. In particular, they obtain Theorem 2.5 and Theorem 2.6. Their results work in the greater generality of \(\FI^m\)-modules over arbitrary Noetherian rings. Li-Yu also probe deeply into the homological algebra of \(\FI^m\)-modules, and in particular study the ``\(\FI^m\)-homology'' of \(\FI^m\)-modules, following Church-Ellenberg's \cite{CE} theory of FI-homology, which we do not explore at all in this paper.

\section{\(\FI^m\)-modules and their properties}
FI is the category introduced by Church-Ellenberg-Farb \cite{CEF} whose objects are finite sets and whose morphisms are injections. Here, we consider the \(m\)-fold product category \(\FI^m\) for some fixed \(m\). Thus, \(\FI^m\) has as objects \(m\)-tuples of finite sets \((S_1, \dots, S_m)\), and morphisms \(f: S \to T\) given by tuples \((f_1, \dots, f_m)\) of injections \(f_i: S_i \hookrightarrow T_i\). This clearly has a skeleton with objects indexed by tuples \((c_1, \dots, c_m)\) of natural numbers, and morphisms \(\mathbf{c} \to \mathbf{d}\) given by tuples \((f_1, \dots, f_m)\) of injections \(f_i: [c_i] \hookrightarrow [d_i]\), where \([n] = \{1, \dots, n\}\). We can define a partial order on such \(\mathbf d\) by saying that \(\mathbf c \le \mathbf d\) if each \(c_i \le d_i\), and then \(\Hom(\mathbf c, \mathbf d) \ne \emptyset\) just when \(\mathbf c \le \mathbf d\). We see that \(\End(\mathbf d) = \Aut(\mathbf d) = S_{\mathbf d} = S_{d_1} \times \cdots \times S_{d_m}\).

An \(\FI^m\)-module is just a functor \(\FI^m \to k\Mod\), where \(k\) is some ring, which we will always take here to be a field of characteristic 0. If \(V\) is an \(\FI^m\)-module, then each \(V_{\mathbf d}\) is an \(S_{\mathbf d}\)-representation. \(\FI^m\)-modules form an abelian category, with maps given by natural transformations of functors. If \(V\) is an \(\FI^m\)-module and \(v_1, \dots v_m \in V\), the submodule generated by the \(v_i\)'s is the smallest submodule that contains them. \(V\) is finitely generated if it is generated by a finite subset. For any \(\mathbf c\), we define the ``free'' module \(M(\mathbf c) = k[\Hom_{\FI^m}(\mathbf c, -)]\). Thus for any \(\mathbf d \ge \mathbf c\), we have \(M(\mathbf c)_{\mathbf d} = k[\Hom_{\FI^m}(\mathbf c, \mathbf d)]\). Furthermore, if \(V\) is an \(\FI^m\)-module and \(v \in V_{\mathbf d}\), there is a natural map \(M(\mathbf{d}) \to V\) taking \(f\) to \(f_* v\), whose image is the submodule of \(V\) generated by \(v\). Thus, \(V\) is finitely generated just when there is a surjection \(\bigoplus_i M(\mathbf{d}_i) \twoheadrightarrow V\).

Gadish \cite{Ga2} proved the Noetherian property for \(\FI^m\)-modules:

\begin{thm}[{\cite[Prop 6.3]{Ga2}}]
If \(V\) is a finitely generated \(\FI^m\)-module, then any submodule is finitely generated.
\end{thm}

\subsection{Projective \(\FI^m\)-modules and \(\FI\sharp^m\)}
Let \(\Res: \FI^n \to S_{\mathbf d}\) be the restriction to a single group. This functor has a left adjoint \(\Ind^{\FI^n}: S_{\mathbf d} \to \FI^n\) given by
\[ \Ind^{\FI^m}(V)_{\mathbf{d} + \mathbf{c}} = \Ind_{S_{\mathbf d} \times S_{\mathbf c}}^{S_{\mathbf{d} + \mathbf{c}}} V \boxtimes k \]
Then \(\Ind^{\FI^n}(V)\) is a projective \(\FI^n\)-module, and tom Dieck \cite[Prop 11.18]{tomD} proved that any projective \(\FI^n\)-module is of this form (in fact, for a large class of category representations). Furthermore, any \(S_{\mathbf d}\)-representation \(V\) is a direct sum of external tensor products of \(S_{d_i}\)-representations, and we have
\begin{align}
\begin{split} 
\Ind^{\FI^m}(V_1 \boxtimes \cdots \boxtimes V_m)_{\mathbf{d} + \mathbf{c}} &= \Ind_{S_{d_1} \times \cdots \times S_{d_m} \times S_{c_1} \times \cdots \times S_{c_m}}^{S_{d_1 + c_1} \times \cdots \times S_{d_m + c_m}} (V_1 \boxtimes \cdots \boxtimes V_m) \boxtimes (k \boxtimes \cdots \boxtimes k) \\
&= (\Ind_{S_{d_1} \times S_{c_1}}^{S_{d_1 + c_1}} V_1 \boxtimes k) \boxtimes (\Ind_{S_{d_m} \times S_{c_m}}^{S_{d_m + c_m}} V_m \boxtimes k) \\
&= \left(\Ind^{\FI}(V_1) \boxtimes \cdots \boxtimes \Ind^{\FI}(V_m)\right)_{\mathbf{d} + \mathbf{c}}
\end{split}
\end{align}
so that any projective \(\FI^m\)-module is a direct sum of tensor product of projective \(\FI^m\)-modules. 

\cite{CEF} give a convenient way of determining when an FI-module is projective: they defined a category \(\FI\sharp\) in which FI embeds, such that \(\FI\sharp\)-modules are exactly the projective FI-modules. Thus a given FI-module is a direct sum of \(\Ind^{\FI}(W)\)'s just when it extends to an \(\FI\sharp\)-module. 

The same construction works for \(\FI^m\): we just take \(\FI\sharp^m\). One description of \(\FI\sharp^m\) is as the category whose objects are those of \(\FI^m\), and whose morphisms \(\mathbf{x} \to \mathbf{y}\) are given by pairs \((\mathbf{z}, f)\), where \(\mathbf{z} \subset \mathbf{x}\) and \(f: \mathbf{z} \hookrightarrow \mathbf{y}\). Thus, it is the category of ``partial morphisms'' of \(\FI^m\). We then have the following.

\begin{thm}
Every \(\FI\sharp^m\)-module is isomorphic to \(\bigoplus_{\mathbf d} \Ind^{\FI^m}(W_{\mathbf d})\) for some representations \(W_{\mathbf d}\) of \(S_{\mathbf d}\).
\end{thm}
\begin{proof}

\cite{CEF} prove that \(\FI\sharp\)-modules are precisely the direct sums of \(\Ind^{\FI}(W)\)'s, and thus that \(\FI\sharp\Mod\) is semisimple. This implies that \((\FI\sharp \times \FI\sharp)\Mod\) is semisimple, and by induction that \(\FI\sharp^m\Mod\) is semisimple, and that its simples are just external tensor products of the simples of \(\FI\sharp\Mod\), which as we said are just the \(\Ind^{\FI}(W)\). The claim follows by (1).

\end{proof}

So \(\FI\sharp^m\)-modules are always sums of tensor products of \(\FI\sharp\)-modules. Note, however, that this is not true for general \(\FI^m\)-modules. Indeed, take the \(\FI^2\)-module \(M\) with
\[ M_{a,b} = \begin{cases} 0 & \text{if } (a,b) \in \{ (0,0), (1,0), (0,1) \} \\ k & \text{otherwise} \end{cases} \]
and where all morphisms of \(\FI^2\) induce the identity \(k \to k\) or the unique map \(0 \to k\). Then \(M\) is not the direct sum of external tensor products of FI-modules. 

First, we know \(M\) cannot be decomposed as the direct sum of two nonzero \(\FI^2\)-modules. Indeed, if we could write \(M = V \oplus V'\), then without loss of generality there would be nonzero elements \(v \in V_{\mathbf c}\), \(v' \in V'_{\mathbf d}\) with \(\mathbf c \le \mathbf d\), and thus no morphism of \(\FI^2\) could have \(f^* v\) be a nonzero multiple of \(v'\). But for any nonzero \(v \in M_{\mathbf c}, v' \in M_{\mathbf d}\) with \(\mathbf c \le \mathbf d\), we have that \(f^* v\) is always a nonzero multiple of \(v'\) for any \(f \in \Hom_{\FI^2}(\mathbf c, \mathbf d)\).

But then if we had \(M = V \boxtimes W\), we would have \(V_1 \boxtimes W_0 = 0, V_0 \boxtimes W_1 = 0, V_1 \boxtimes W_1 = k\), which is impossible.

Finally, (1) lets us compute the character of an \(\FI\sharp^m\)-module. First, define a \emph{character polynomial for \(\FI^m\)} to be a polynomial in \(k[X_1^{(1)}, \dots X_1^{(m)}, X_2^{(1)}, \dots]\), where \(X_i^{(k)}\) is the class function on \(S_{\mathbf d}\) that counts the number of \(i\)-cycles in \(S_{d_k}\). We then have the following.

\begin{prop}
If \(V\) is a finitely generated \(\FI\sharp^m\)-module, then \(\chi_{V_n}\) is given by a single character polynomial \(P\) for all \(n\).
\end{prop}
\begin{proof}
If \(V = \Ind^{\FI^m}(V_1 \boxtimes \cdots \boxtimes V_m)\), then by (1), \(V = \Ind^{\FI}(V_1) \boxtimes \cdots \boxtimes \Ind^{\FI}(V_m)\). By \cite[Thm 4.1.7]{CEF}, the character of \(\Ind^{\FI}(V_i)_n\) is given by a single character polynomial \(P_i \in k[X^{(i)}_1, X^{(i)}_2, \dots]\) for all \(m\). Then \(\chi_V\) = \(P_1 \cdots P_m\). By Theorem 2.2, a general \(\FI\sharp^m\)-module is a direct sum of such \(V\)'s, so the claim follows.
\end{proof}

\subsection{Shift functors and representation stability}
Another basic operation on \(\FI^m\)-modules are the shift functors. For \(\mathbf a \in \FI^m\), let
\[S_{+\mathbf{a}} : \FI^m\Mod \to \FI^m\Mod\] be the functor defined by \(S_{+\mathbf{a}}(V)_{\mathbf{d}} = V_{\mathbf{d} + \mathbf{a}}\). Following \cite{CEFN} and \cite{Na}, we will use this functor to establish representation stability for \(\FI^m\).

Notice that we have \(S_{+\mathbf{a}} \circ S_{+\mathbf{b}} = S_{+\mathbf{b}} \circ S_{+\mathbf{a}} = S_{+(\mathbf{a} + \mathbf{b})}\). In particular, if we decompose \(\mathbf{a}\) into ``unit vectors'' as \(\mathbf{a} = a_1 \mathbf{e}_1 + \cdots + a_n \mathbf{e}_n\), where \(\mathbf{e}_i = (0,\dots, 1_i, \dots, 0)\), then \(S_{\mathbf{a}} = (S_{+\mathbf{e}_n})^{a_n} \circ \cdots \circ (S_{+\mathbf{e}_1})^{a_1} \). The following fundamental proposition describes the effect of shift functors on the ``free'' modules \(M(\mathbf{d})\), generalizing \cite[Prop 2.12]{CEFN}:

\begin{prop}
For any \(\mathbf{a}, \mathbf{d} \in \FI^m\), there is a natural decomposition
\[ S_{+\mathbf{a}}M(\mathbf{d}) = M(\mathbf{d}) \oplus Q_a \]
where \(Q_a\) is a free \(\FI^m\)-module generated in degree \(\le d - 1\).
\end{prop}
\begin{proof}
It is enough to prove this for the case \(\mathbf{a} = \mathbf{e}_i\), since for a general \(\mathbf{a}\), we know that \(S_{+\mathbf{a}}\) is a composition of the \(S_{+\mathbf{e}_i}\)'s. A basis for \(S_{+\mathbf{e}_i} M(\mathbf{d})_{\mathbf{c}}\) is the set of tuples of injections \(\mathbf{f}\), where
\begin{gather*}
f_1: [d_1] \hookrightarrow [c_1] \\
\vdots \\
f_i: [d_i] \hookrightarrow [c_1] \sqcup \{\star\} \\
\vdots \\
f_m: [d_m] \hookrightarrow [c_m]
\end{gather*}
This set can be partitioned into \(d_i + 1\) subsets, according to \(f_i^{-1}(\star)\)---that is, by which element of \([d_i]\) (or possibly none) gets mapped by \(f_i\) to \(\star\). Notice that \(f_i^{-1}(\star)\) is not affected by post-composing with an \(\FI\)-morphism. Thus this partition actually defines a decomposition of \(S_{+\mathbf{e}_i} M(\mathbf{d})\) as a direct sum of \(\FI^m\)-modules.

For \(T \subset [d_i]\) of size at most 1, let \(M^T\) be the submodule of \(S_{+\mathbf{e}_i} M(\mathbf{d})\) spanned by those \(\mathbf{f}\) with \(f_i^{-1}(\star) = T\). These \(\mathbf{f}\) are distinguished by the restrictions \(f|_{\mathbf{d} - T}\), and we have \((g_* f)|_{\mathbf{d} - T} = g \circ f|_{\mathbf{d} - T}\). We therefore have \(M^\emptyset \cong M(\mathbf d)\), and 
\[M^{\{t\}} \cong M(\mathbf d - \mathbf{e}_i) = M(d_1, \dots, d_i - 1, \dots, d_m).\]
So we have a decomposition
\[ S_{+\mathbf{e}_i} M(\mathbf{d}) = M^\emptyset \oplus \bigoplus_{t \in [d_i]} M^{\{t\}} = M(\mathbf{d}) \oplus \bigoplus_{t \in [d_i]} M(\mathbf{d} - \mathbf{e}_i). \]
\end{proof}

Following \cite{Na}, say that a finitely generated \(\FI^m\)-module \(V\) is \emph{filtered} if it admits a surjection
\[ \Pi: \bigoplus_{i=1}^g M(\mathbf{d}_i) \twoheadrightarrow V\]
such that the filtration \(0 = V^0 \subset V^1 \subset \dots \subset V^g = V\) given by
\[ V^r := \Pi\left( \bigoplus_{i=1}^r M(\mathbf{d}_i)\right),\: 0 \le r \le d \]
has successive quotients \(V^r/V^{r-1}\) which are projective \(\FI^m\)-modules.

\begin{thm} \thlabel{shift_filtered}
For any finitely generated \(\FI^m\)-module \(V\), there is some \(\mathbf a \in FI^m\) such that \(S_{+\mathbf a} V\) is filtered. 

Furthermore, there are filtered \(\FI^m\)-modules \(J^0, \dots, J^N\) and a sequence
\[ 0 \to V \to J^0 \to \cdots \to J^N \to 0 \]
which is exact in high enough degree. That is, the sequence
\[ 0 \to V_{\mathbf{d}} \to J^0_{\mathbf{d}} \to \cdots \to J^N_{\mathbf{d}} \to 0 \]
is exact for sufficiently large \(\mathbf{d}\).
\end{thm}
Our proof follows the one given by Nagpal \cite[Thm A]{Na} and Ramos \cite[Thm 3.1]{Ra} for the case \(m = 1\). As we mentioned in the introduction, \thref{shift_filtered} was independently proven by Li-Yu \cite[Thm 1.5, Thm 4.10]{LY}.
\begin{proof}
Let \(V\) be an \(\FI^m\)-module generated in degree \(\mathbf{D}\) and related in degree \(\mathbf{r}\). This means there is an exact sequence
\[ 0 \to K \to \bigoplus_{i=1}^g M(\mathbf{d}_i) \to V \to 0 \]
where each \(\mathbf{d}_i \le \mathbf{D}\). We put \(\widetilde{V} := \bigoplus_{i=1}^g M(\mathbf{d}_i) \). Notice that by Proposition 2.3, we may write
\[ S_{+\mathbf{a}} \widetilde{V} = \bigoplus_{i=1}^g M(\mathbf{d}_i) \oplus Q_i^{\mathbf{a}}\]
where each \(Q_i^{\mathbf{a}}\) is free and generated in degree \(< \mathbf{d}_i\). We have the following commutative diagram with exact rows and columns:
\[ \begin{diagram}
0 & \rTo & \widetilde{U}^{\mathbf{a}} & \rTo & S_{+\mathbf{a}} \widetilde{V} & \rTo & M(\mathbf{d}_i) & \rTo & 0 \\
&& \dTo^{\Pi^{\mathbf{a}}} && \dTo^{S_{+\mathbf{a}} \Pi} && \dTo^{\phi^{\mathbf{a}}} && \\
0 & \rTo & U^{\mathbf{a}} & \rTo & S_{+\mathbf{a}} V & \rTo & A^{\mathbf{a}} & \rTo & 0 \\
&& \dTo && \dTo && \dTo && \\
&& 0 && 0 && 0 && \\
\end{diagram} \]
Looking at the \(\FI^m\)-module \(K^{\mathbf{a}} = \ker(\phi^{\mathbf{a}})\), we observe that the \(S_{\mathbf{d}_i}\)-modules \(K^{\mathbf{a}}_{\mathbf{d}_i}\) are increasing in \(\mathbf{a}\), and therefore must stabilize for large \(\mathbf{a}\). In fact, we can take \(\mathbf{a} = \mathbf{r}\).

Fixing such an \(\mathbf{a}\), it must be the case that \(K^{\mathbf{a}}\) is generated in degree \(\le \mathbf{d}_i\). By exactness of \(\Ind^{\FI^m}\), it follows that \(A^{\mathbf{a}} = \Ind^{\FI^m}(W)\) for some \(S_{\mathbf{d}_i}\)-module W. So we are left with the exact sequence
\[ 0 \to U^{\mathbf{a}} \to S_{+\mathbf{a}} V \to \Ind^{\FI^m}(W) \to 0 \]
By induction on degree, \(S_{+\mathbf{b}} U^{\mathbf{a}}\) is filtered for sufficiently large \(\mathbf{b}\). Since shifting is exact, we obtain
\[ 0 \to S_{+\mathbf{b}} U^{\mathbf{a}} \to S_{+(\mathbf{a}+\mathbf{b})} V \to S_{+\mathbf{b}} \Ind^{\FI^m}(W) \to 0 \]
We conclude that \(S_{+(\mathbf{a}+\mathbf{b})} V\) must be filtered. This completes the first part of the theorem.

For the second part, let \(\mathbf{a}\) be large enough so that \(S_{+\mathbf{a}} V\) is filtered. Continuing the notation of the first part, we have
\[ S_{+\mathbf{a}} \widetilde{V} = \bigoplus_i M(\mathbf{d}_i) \oplus \widetilde{Q}\]
where \(\widetilde{Q}\) is generated in degree \(< D\). We thus have an exact sequence
\[ 0 \to V \to S_{+\mathbf{a}} V \to Q \to 0\]
where \(Q\) is generated in degree \(< D\). By induction, the claim is true for \(Q\), say with filtered modules \(K^0, \dots, K^M\). If we form the sequence
\[ 0 \to V \to S_{+\mathbf{a}} V \to K^0 \to \cdots K^M \to 0\]
the claim then follows.
\end{proof}
Recall that the irreducible representations of \(S_{\mathbf{d}}\) are just given by tensor products of irreducible representations of each \(S_{d_i}\), which are indexed by partitions of \(d_i\). If \(\boldsymbol{\lambda} = (\lambda_1, \dots, \lambda_m)\) is a list of partitions of \(\mathbf{d} = (d_1, \dots, d_m)\), then write \(\Irr(\boldsymbol{\lambda}) = \Irr(\lambda_1) \boxtimes \cdots \boxtimes \Irr(\lambda_m)\) for the irreducible representation indexed by \(\boldsymbol \lambda\). Extend \(\boldsymbol{\lambda}\) to \(\mathbf{c} \ge \mathbf{d} + (\lambda_1^{(1)}, \dots, \lambda_m^{(1)})\) as follows:
\[ \boldsymbol{\lambda}[\mathbf{c}] = ((c_1 - |\lambda_1|, \lambda_1), \dots, (c_m - |\lambda_m|, \lambda_m)) \]
Then we obtain the following.
\begin{thm}[\bfseries Representation stability for \(\FI^m\)]
Let \(V\) be a finitely-generated \(\FI^m\)-module. Then there is a character polynomial \(P\)  such that for all \(\mathbf{d} \gg 0\), the character \(\chi_{V_{\mathbf{d}}} = P_{\mathbf d}\). In particular, the dimension \(\dim V_{\mathbf{d}}\) is eventually given by a polynomial in the \(d_i\)'s. Furthermore, the decomposition into irreducibles has multiplicities independent of \(\mathbf{d}\) for \(\mathbf{d}\) large:
\[ V_{\mathbf{d}} = \bigoplus_{\boldsymbol{\lambda}} \Irr(\boldsymbol{\lambda[\mathbf{d}]})^{c_{\boldsymbol \lambda}} \; \text{for all } \mathbf{d} \gg 0 \]
\end{thm}
As we mentioned in the introduction, the claim about multiplicity stability was independently proven by Li-Yu \cite[Thm 1.8]{LY}.
\begin{proof}
Theorem 2.5 gives us filtered \(\FI^m\)-modules \(J^0, \dots, J^N\) and a sequence
\[ 0 \to V \to J^0 \to \cdots \to J^N \to 0\]
which is exact in high enough degree. By semisimplicity of \(\mathbb{Q}[S_{\mathbf d}]\), it is therefore enough to prove the claim for the \(J^i\). But since each \(J^i\) is filtered, meaning it has a filtration whose graded pieces are projective, we can reduce to the case of \(\FI\sharp^m\)-modules, again by semisimplicity. By Proposition 2.3, the character of \(\FI\sharp^m\)-modules is given by a single character polynomial. Finally, by \cite[Prop 3.26]{CEF} we know that each \(\Ind^{\FI}(V_i)\) satisfies representation stability, so since an \(\FI\sharp^m\)-module is just a direct sum of tensor products of these, it therefore satisfies representation stability.
\end{proof}

\subsection{Tensor products and \(\FI^m\)-algebras}
Here we proceed to generalize the notions introduced in \cite[\S4.2]{CEF} from FI to \(\FI^m\). Given \(\FI^m\)-modules \(V\) and \(V'\), their tensor product \(V \otimes V'\) is the \(\FI^m\)-module with \((V \otimes V')_{\mathbf d} = V_{\mathbf d} \otimes V'_{\mathbf d}\), where \(\FI^m\) acts diagonally.

A \emph{graded \(\FI^m\)-module} is a functor from \(\FI^m\) to graded modules, so that each piece is graded, and the induced maps respect the grading. If \(V\) is graded, each graded piece \(V^i\) is thus an \(\FI^m\)-module. If \(V\) and \(W\) are graded, their tensor product \(V \otimes W\) is graded in the usual way. Say that \(V\) is \emph{finite type} if each \(V^i_n\) is finitely generated.

Similarly, an \(\FI^m\)-algebra is a functor from \(\FI^m\) to \(k\)-algebras, which can also be graded. Here our algebras will always be graded-commutative. We can also define graded co-\(\FI^m\)-modules and algebras, as functors from \((\FI^m)^{\text{op}}\). A (co-)\(\FI^m\) algebras is \emph{generated (as an \(\FI^m\)-algebra)} by a submodule \(V\) when each \(A_{\mathbf d}\) is generated as an algebra by \(V_{\mathbf d}\).

Finally, there is another type of tensor product that we will need. Suppose \(V\) is a graded vector space with \(V^0 = k\). Then the space \(V^{\otimes \bullet}\) defined by \((V^{\otimes \bullet})_{\mathbf d} = \boxtimes_{i = 1}^m V^{\otimes d_i}\) has the structure of an \(\FI\sharp^m\)-module, as in \cite[Defn 4.2.5]{CEF}, with the morphisms permuting and acting on the tensor factors.

The following theorem characterizes the above constructions.

\begin{thm} \ \begin{enumerate}
\item If \(V\) and \(V'\) are finitely generated \(\FI^m\)-modules, then \(V \otimes V'\) is finitely generated.
\item Let \(A\) be a graded \(\FI^m\)-algebra generated by a graded submodule \(V\), where \(V^0 = 0\). If \(V\) is finite type, then \(A\) is finite type.
\item Let \(V\) be a graded vector space with \(V^0 = k\). If \(V\) is finite type as a graded vector space, then \(V^{\otimes \bullet}\) is finite type as a graded \(\FI^m\)-module.
\item Let \(X\) be a connected space such that \(H^*(X; k)\) is finite type. Then \(H^*(X^\bullet; k)\) is an \(\FI\sharp^m\)-algebra of finite type.
\end{enumerate}
\end{thm}
\begin{proof} \ \begin{enumerate}
\item It is enough to prove the theorem in the case where \(V\) and \(W\) are projective. But this is \cite[Thm B(1)]{Ga2}.
\item Let \(T(V)\) be the tensor algebra on \(V\). It is of finite-type, since
\[ (T(V))^k = \bigoplus_{i_1 + \cdots + i_m = k} V^{i_1} \otimes \cdots \otimes V^{i_m} \]
where each summand on the right is finitely generated by (1). But since \(A\) is an \(\FI^m\)-algebra generated by \(V\), there is an \(\FI^m\)-algebra surjection \(T(V) \twoheadrightarrow A\). Therefore \(A\) is of finite type.
\item We have \((V^{\otimes \bullet})_{\mathbf d} = (V^{\otimes \bullet})_{d_1} \boxtimes \dots \boxtimes (V^{\otimes \bullet})_{d_m}\). Since the \(\FI\sharp\)-algebra \(V^{\otimes \bullet}\) is finitely generated \cite[Prop 4.2.7]{CEF}, we conclude that the \(\FI\sharp^m\)-algebra \(V^{\otimes \bullet}\) is finitely generated.
\item As \cite[Prop 6.1.2]{CEF} explain for the case \(m = 1\), this essentially follows from (3) and the K\"unneth formula; technically sometimes a sign is introduced when permuting the order of tensor factors, but this does not change the proof of (3). The degree 0 part is \(k\) by connectivity.
\end{enumerate}
\end{proof}

\section{Spaces of 0-cycles}
Let \(M\) be a manifold, \(n\) an integer, and \(\mathbf{d} = (d_1, \dots, d_m)\). Define \(M^{\mathbf d} := \prod_i M^{d_i}\). Then \(M^{\bullet}\) forms a co-\(\FI^m\)-space, where an \(\FI\sharp^m\) morphism \(\mathbf{f}: \mathbf{c} \hookrightarrow \mathbf{d}\) acts on \((v_1, \dots, v_m) \in M^{\mathbf d}\) by
\[ f^*(v_1, \dots, v_m) = (f_1^*(v_1), \dots, f_m^*(v_m)) \]
where the action of \(f: [c] \hookrightarrow [d]\) on \(v = (m_1, \dots, m_d) \in M^d\) is the usual co-FI action:
\[ f^* (m_1, \dots m_d) = (m_{f(1)}, \dots, m_{f(c)}) \]
Therefore by applying the contravariant functor of taking cohomology, we obtain an \(\FI^m\)-algebra \(H^*(M^\bullet)\). We proved in Theorem 2.7.4 that \(H^*(M^\bullet)\) is of finite type. Next, define
\[ \widetilde{\mathcal{Z}}^{\mathbf d}_n(M) = \left\{(v_1, \dots, v_m) \in \prod_i M^{d_i} \mathrel{\bigg|} \text{ no } m \in M \text{ appears } n \text{ or more times in each } v_i\right\} \]
Then \(\widetilde{\mathcal{Z}}^{\bullet}_n(M)\) is a co-\(\FI^m\)-subspace of \(M^{\bullet}\). Indeed if \(\mathbf{f}: \mathbf{c} \hookrightarrow \mathbf{d}\) is an \(\FI^m\)-morphism and \((v_1, \dots, v_m) \in \widetilde{\mathcal{Z}}^{\mathbf d}_n(M)\), then \(f^*(v_1, \dots, v_m) = (f_1^*(v_1), \dots, f^*_m(v_m)) \in \widetilde{\mathcal{Z}}^{\mathbf c}_n(M)\). This holds because the coordinates of \(f_i^*(v_i)\) are just drawn from the coordinates of \(v_i\), so the number of times any \(m \in M\) appears in \(f_i^*(v_i)\) is bounded by the number of times \(m\) appears in \(v_i\).

Thus we obtain an \(\FI^m\)-algebra \(H^*(\widetilde{\mathcal{Z}}^\bullet_n(M))\). Our first main theorem is the following. 

\begin{thm}
\thlabel{top_stab}
Let \(k\) be a field and let \(M\) a connected, oriented manifold of dimension at least 2 with \(\dim H^*(M; k) < \infty\). Then the \(\FI^m\)-algebra \(H^*(\widetilde{\mathcal{Z}}^\bullet_n(M); k)\) is of finite type.
\end{thm}
In the case of usual configuration space, where \(m=1\) and \(n = 2\), \thref{top_stab} was proven by Church-Ellenberg-Farb \cite[Thm 6.2.1]{CEF}, based on an analysis of the Leray spectral sequence studied by Totaro \cite{To}.
\begin{proof}
Following \cite[\S5]{FWW}, the basic strategy for understanding \(H^*(\widetilde{\mathcal{Z}}^\bullet_n(M); k)\) is to analyze the Leray spectral sequence associated to the inclusion \(\widetilde{\mathcal{Z}}^{\mathbf d}_n(M) \hookrightarrow M^{\mathbf d}\). This is a spectral sequence with \(E_2\) page
\[ E^{p,q}_{2,\mathbf{d}} = H^p(M^{\mathbf{d}}; U \mapsto H^q(U \cap \widetilde{\mathcal{Z}}^{\mathbf d}_n(M))) \]
and converging to \(H^*(\widetilde{\mathcal{Z}}^\bullet_n(M))\). Since, as we saw, the inclusion \(\widetilde{\mathcal{Z}}^{\bullet}_n(M) \hookrightarrow M^{\bullet}\) is a map of co-\(\FI^m\)-spaces, and the Leray spectral sequence is functorial, we actually obtain a spectral sequence \(E^{p,q}_2\) of \(\FI^m\)-modules.

Now, the bottom row \(E^{*,0}_2\) is just isomorphic to \(H^*(M^{\bullet})\) as an \(\FI^m\)-module. Furthermore, as Farb-Wolfson-Wood prove \cite[Thm 5.6]{FWW}, the leftmost column \(E^{0,*}_2\) is isomorphic to \(H^*(\widetilde{\mathcal{Z}}_n(\mathbb{R}^N))\), where \(N = \dim M\). Furthermore, they show that the \(E_2\) page is generated as an \(\FI^m\)-algebra by \(E^{*,0}_2\) and \(E^{0,*}_2\). As we showed in Proposition 3.1, \(E^{*,0}_2\) is finite-type. And \(\widetilde{\mathcal{Z}}_n(\mathbb{R}^N)\) is a subspace arrangement of the type studied by Gadish \cite{Ga1}. He proves \cite[Thm B]{Ga1} that \(H^*(\widetilde{\mathcal{Z}}_n(\mathbb{R}^N))\) is a finite-type \(\FI\sharp^m\)-module. Since it is generated as an algebra by a finite-type \(\FI^m\)-module, the \(E_2\) page as a whole is a finite-type \(\FI^m\)-module. The \(E_\infty\) page is a subquotient of the \(E_2\) page, so by Noetherianity it is finite type, and therefore \(H^*(\widetilde{\mathcal{Z}}^\bullet_n(M))\) is finite type.
\end{proof}

\begin{proof}[Proof of Theorem 1.1]
This is a direct corollary of Theorem 3.2 after applying Theorem 2.6
\end{proof}

Next, as in \cite{CEF}, if \(M\) is the interior of a compact manifold with boundary, we obtain the following generalization of \cite[Prop 6.1.2]{CEF}.

\begin{prop}
Let \(M\) be the interior of a connected compact manifold \(\overline{M}\) with nonempty boundary \(\partial \overline{M}\). Then \(\widetilde{\mathcal{Z}}_n(M)\) has the structure of a homotopy \(\FI\sharp^m\)-space, that is, a functor \(\FI\sharp^m \to \mathrm{hTop}\), the category of spaces and homotopy classes of maps.
\end{prop}
\begin{proof}
We follow the argument in \cite{CEF}. Fix a collar neighborhood \(R\) of one component \(\partial \overline{M}\), and fix a homeomorphism \(\Phi: M \cong M - \overline{R}\) isotopic to the identity. For any \(m\)-tuple of inclusions of \(m\)-tuples of finite sets \(\mathbf{X} \subset \mathbf{Y}\), define a map
\[ \Psi^Y_X: \widetilde{\mathcal{Z}}^{\mathbf{X}}_n(M) \to \widetilde{\mathcal{Z}}^{\mathbf{Y}}_n(M) \]
up to homotopy, as follows. First, if \(\mathbf{Y} = \mathbf{X}\), set \(\Psi^{\mathbf Y}_{\mathbf X} = \id\). Next, note that \(\Conf_{\sqcup X_i}(M) \hookrightarrow \widetilde{\mathcal{Z}}^{\mathbf X}_n(M)\). So fix an embedding \(q^{\mathbf Y}_{\mathbf X}: \sqcup_i ( Y_i - X_i) \hookrightarrow R\) of \(\Conf_{\sqcup Y_i - X_i}(M)\). Then any element \(f : \sqcup X_i \to M\) in \(\widetilde{\mathcal{Z}}^{\mathbf X}_n(M)\) extends to a map \(\Psi^{\mathbf Y}_{\mathbf X}(f): \sqcup Y_i \to M\) by
\[
\Psi^{\mathbf Y}_{\mathbf X}(f)(t) = \begin{cases} \Psi(f(t)) & t \in \sqcup_i X_i \\ q^{\mathbf Y}_{\mathbf X}(t) & t \notin \sqcup_i X_i \end{cases}
\]
The image of \(\Phi\) is disjoint from \(R\), while the image of \(q^{\mathbf Y}_{\mathbf X}\) is contained in \(R\), so the above map does not have any more coincidences of points than \(f\) itself did, and therefore \(\Psi^{\mathbf Y}_{\mathbf X}(f)\) does give an element of \(\widetilde{\mathcal{Z}}^{\mathbf Y}_n(M)\). Furthermore, since \(\Conf_{\sqcup Y_i - X_i}(M)\) is connected (since \(R\) is, and \(\dim R \ge 2\)), different choices of \(q^{\mathbf Y}_{\mathbf X}\) give homotopic maps, so \(\Phi^{\mathbf Y}_{\mathbf X}\) is well-defined up to homotopy.

Now, an \(\FI\sharp^m\) morphism \(\mathbf Z \to \mathbf Y\) consists of a map \(\mathbf X \hookrightarrow \mathbf{Z}\) and a map \(X \hookrightarrow  Y\). Normally if we were extending from an \(\FI^m\)-structure, we would think of \(\mathbf X\) as being a subset of \(\mathbf Z\), but since we are extending from a \emph{co}-\(\FI^m\)-structure, it is more natural to think of \(\mathbf X\) as a subset of \(\mathbf Y\), with an explicit map \(\mathbf{a}: \mathbf{X} \to \mathbf{Z}\). The induced map is then given by
\begin{align*}
\widetilde{\mathcal{Z}}^{\mathbf Z}_n(M) \to \widetilde{\mathcal{Z}}^{\mathbf X}_n(M) & \xrightarrow{\Psi^{\mathbf Y}_{\mathbf X}} \widetilde{\mathcal{Z}}^{\mathbf Y}_n(M) \\
(v_1, \dots, v_m) \mapsto (a^*_1(v_1), \dots, a^*_m(v_m))&
\end{align*}
It is straightforward to verify that this is functorial up to homotopy, as \cite[Prop 6.4.2]{CEF} do for \(m = 1\).
\end{proof}
In particular, when the conditions of Theorem 3.2 hold, then \(H^*(\widetilde{\mathcal{Z}}_n(M))\) is an \(\FI\sharp^m\)-module. We therefore obtain the following.

\begin{cor}
Let \(M\) be a connected orientable manifold of dimension at least 2 which is the interior of a compact manifold with nonempty boundary. Then for each \(i\), the characters of the \(S_{\mathbf d}\)-representations \(H^i(\widetilde{\mathcal{Z}}^{\mathbf d}_n(M; \mathbb{Q})\) are given by a single character polynomial for \emph{all} \(\mathbf{d}\).
\end{cor}

\section{\'Etale representation stability of \(\widetilde{\mathcal{Z}}^{\mathbf d}_n(X)\)}
In this section we consider the case where we replace the manifold \(M\) with a scheme \(X\) over \(\mathbb{Z}[1/N]\), as in Farb-Wolfson \cite{FW2}. Here the \(\mathbb{C}\)-points \(X(\mathbb{C})\) take the place of \(M\). However, now we can also consider the points \(X(\mathbb{F}_q)\) over a finite field \(\mathbb{F}_q\).

So, if \(X\) is a scheme over \(\mathbb{Z}[1/N]\), define \(\widetilde{\mathcal{Z}}^{\mathbf d}_n(X)\) as the functor of points
\[ \widetilde{\mathcal{Z}}^{\mathbf d}_n(X)(R) = \left\{ (v_1, \dots, v_m) \in \prod_i X(R)^{d_i} \mathrel{\bigg|} \text{ no } x \in X(R) \text{ appears } n \text{ or more times in each } v_i\right\} \]
for any \(\mathbb{Z}[1/N]\)-algebra \(R\). Thus \(\widetilde{\mathcal{Z}}^{\mathbf d}_n(X)\) has the structure of a scheme over \(\mathbb{Z}[1/N]\). If \(X\) is smooth, then \(\widetilde{\mathcal{Z}}^{\mathbf d}_n(X)\) is smooth. Treated as a single object, \(\widetilde{\mathcal{Z}}^\bullet_n(X)\) forms a co-\(\FI^m\)-scheme. 

Let \(K\) be a number field or an finite field over \(\mathbb{Z}[1/N]\), and let \(l\) be a prime number invertible in \(K\). Then we can base-change to \(\overline{K}\) and consider the etale cohomology \(H^i_{\acute et}(\widetilde{\mathcal{Z}}^{\mathbf d}_n(X)_{/\overline{K}}, \mathbb{Q}_l)\). Thus \(H^i_{\acute et}(\widetilde{\mathcal{Z}}^{\bullet}_n(X)_{/\overline{K}}, \mathbb{Q}_l)\) is an \(\FI^m\)-module equipped with an action of \(\Gal(\overline{K}/K)\) commuting with the \(\FI^m\) action. Following \cite{FW2}, we call such an object a \emph{\(\Gal(\overline{K}/K)\)-\(\FI^m\)-module}.

We would like to prove Theorem 1.2, that \(H^i_{\acute et}(\widetilde{\mathcal{Z}}^{\bullet}_n(X)_{/\overline{K}}, \mathbb{Q}_l)\) is finitely generated as an \(\Gal(\overline{K}/K)\)-\(\FI^m\)-module. One method of doing this, the one used in Farb-Wolfson \cite{FW2} and Casto \cite{Ca2}, would be to find an appropriate compactification of \(\widetilde{\mathcal{Z}}^{\mathbf d}_n(X)\), use this to argue that the \'etale cohomology \(H^i_{\acute et}(\widetilde{\mathcal{Z}}^{\mathbf d}_n(X)_{/\overline{K}}, \mathbb{Q}_l)\) is isomorphic to the singular cohomology \(H^i(\widetilde{\mathcal{Z}}^{\mathbf d}_n(X)(\mathbb{C}), \mathbb{Q}_l)\) of the complex points, and then conclude by Theorem 3.2. We are confident that this approach could be made to work. However, constructing and proving the requisite properties about the desired compactification would involve a fair amount of technical work that we want to avoid. (In the special case of configuration spaces, this technical work was done by Fulton-Macpherson \cite{FM}.) Instead, we will reprove Theorem 3.2 for \'etale cohomology, by directly computing with the same Leray spectral sequence for the inclusion \(\widetilde{\mathcal{Z}}^{\bullet}_n(X) \hookrightarrow X^\bullet\).

\begin{proof}[Proof of Theorem 1.2]
This argument is essentially the one given in the proof of the second part of \cite[Thm 5.6]{FWW}, on page 28. However, Farb-Wolfson-Wood do not quite state the conclusion in terms of \'etale cohomology, so we will redo it. Given a variety \(X\), define the ``big diagonal''
\[ \Delta^{\mathbf d}_n(X) := \left\{ (v_1, \dots, v_m) \in \prod_i X^{d_i} \mathrel{\bigg|} \text{ some } x \in X \text{ appears at least } n \text{ times in each } v_i\right\}\]
so that \(\widetilde{\mathcal{Z}}^{\mathbf d}_n(X) = X^{\mathbf d} - \Delta^{\mathbf d}_n(X)\). 
Let \(m := \dim X\). By \cite[Lem 16.8]{Mi}, the pair \((\Delta^{\mathbf d}_n(X), X^{\mathbf d})\) is locally isomorphic (for the \'etale topology) to the pair \((\Delta^{\mathbf d}_n(\mathbb{A}^m), (\mathbb{A}^m)^{\mathbf d})\). Indeed, we can describe this isomorphism explicitly. Choose regular functions \(f_1, \dots, f_m\) defined on an open neighborhood \(V\) in \(X\), giving an \'etale map \(F: V \to \mathbb{A}^m\). This gives us an \'etale map \(F^{\mathbf d}: V^{\mathbf d} \to (\mathbb{A}^m)^{\mathbf d}\), and under this map, the image of \(\Delta^{\mathbf d}_n(X)\) is \(\Delta^{\mathbf d}_n(\mathbf{A}^m)\).

The point is the following: in the classical topology (as in \cite{To}) we were able to argue about the sheaf \(U \mapsto H^q(U \cap \widetilde{\mathcal{Z}}^{\mathbf d}_n(X))\) at a point \(x\) in the big diagonal by picking an open set \(U\) small enough that it only intersects the irreducible component of \(\Delta^{\mathbf d}_n(X)\) containing \(x\). In the \'etale topology, we do not have enough fine-grained control over neighborhoods to find one that only intersects one component. However, the argument still goes through, because the point is that, as we have just seen, we can find an \'etale neighborhood \(V\) of \(x\) that \'etale-locally looks like \(\widetilde{\mathcal{Z}}^{\mathbf d}_n(\mathbb{A}^m)\), and this is enough.

Indeed, from here we can basically follow the proof of the first part of \cite[Thm 5.6]{FWW}. As Farb-Wolfson-Wood say (and using their notation), it is enough to give an \(S_{\mathbf d}\)-equivariant isomorphism of sheaves
\begin{equation} \label{iso_shvs}
R^q j_{X*} \mathbb{Z} \cong \bigoplus_{I \in \Pi^{\mathbf d}_n} \epsilon_I(q)
\end{equation}
where \(j_X: \widetilde{\mathcal{Z}}^{\mathbf d}_n(X) \hookrightarrow X^{\mathbf d}\) and
\[ \epsilon_I(q) := \tilde{H}_{\text{cd}(I,X) - q - 2}(\Delta(\overline{\Pi^{\mathbf d}_n(\le I)}); \mathbb{Z}) \otimes \text{coor}(X_I) \]
But by restricting to the \'etale neighborhoods \(V\) mentioned above, we obtain for each \(x \in \Delta^{\mathbf d}_n(X)\), an isomorphism of stalks
\[ (R^q j_{X*} \mathbb{Z})_x \cong (R^q j_{\mathbb{A}^m*} \mathbb{Z})_y \]
where the right hand side denotes the stalk at a generic \(y\) in the component of \(\Delta^{\mathbf d}_n(\mathbb{A}^m)\) containing \(F(x)\). This isomorphism of stalks is explicitly mentioned on \cite[p. 28]{FWW}. But now, since \cite{FWW} already verified (\ref{iso_shvs}) for \(\mathbb{A}^m\), we conclude that it holds for \(X\).

Our argument from the proof of Theorem 3.1 therefore applies directly, since we have reproven the necessary tools of \cite[Thm 5.6]{FWW}. To wit, we have just shown that the \(E_2\) page of the spectral sequence is generated as a \(\Gal(\overline{K}/K)\)-\(\FI^m\)-algebra by \(H^*_{\acute et}(X^{\bullet})\) and \(H^*_{\acute et}(\widetilde{\mathcal{Z}}_n(\mathbb{A}^d))\). Again, the first is finite-type by Proposition 3.1, and second is finite-type by \cite[Thm B]{Ga2}---note that Gadish specifically addresses the \'etale cohomology and \(\Gal(\overline{K}/K)\)-action of this subspace arrangement. Thus the \(E_2\) page as a whole is a finite-type \(\Gal(\overline{K}/K)\)-\(\FI^m\)-module. So again by Noetherianity, we conclude that \(H^*_{\acute et}(\widetilde{\mathcal{Z}}^{\bullet}_n(X))\) is finite type.
\end{proof}

\section{Convergence}
Recall from Theorem 2.6 that if \(V\) is a finitely generated \(\FI^m\)-module, then the characters \(\chi_{V_{\mathbf d}}\) are eventually given by a single character polynomial for all large \(\mathbf d\). Furthermore, for another character polynomial \(P\), we know that the inner product \(\langle P_{\mathbf d}, V_{\mathbf d}\rangle_{S_{\mathbf d}}\) is eventually independent of \(\mathbf d\). We put
\[ \langle P, V \rangle = \lim_{\mathbf d \to \infty} \langle P_{\mathbf d}, V_{\mathbf d} \rangle_{S_{\mathbf d}} \]
for the limiting multiplicity.

Thus if \(Z\) is a co-\(\FI^m\)-scheme that satisfies \'etale representation stability, there is a stable inner product \(\langle P, H^i_{\acute et}(Z; \mathbb{Q}_l) \rangle\), eventually independent of \(\mathbf d\). However, for our applications we need to bound how these inner products grow in \(i\). The following proposition is helpful in doing so:

\begin{prop}
For any graded \(\FI^m\)-module \(V^*\), the following are equivalent: \begin{enumerate}
\item For each character polynomial \(P\), \(|\langle P_{\mathbf d}, V^i_{\mathbf d}\rangle|\) is bounded subexponentially in \(i\) and uniformly in \(\mathbf d\).

\item For every \(\mathbf a\), the dimension \(\dim\left((V^i_{\mathbf d})^{S_{\mathbf d - \mathbf a}}\right)\) is bounded subexponentially in \(i\) and uniformly in \(\mathbf d\).
\end{enumerate}
\end{prop}
\begin{proof}
First, note that \(\dim\left((V^i_{\mathbf d})^{S_{\mathbf d - \mathbf a}}\right) = \langle M(\mathbf a)_{\mathbf d}, V_{\mathbf d} \rangle_{S_{\mathbf d}}\). For any irreducible \(S_{\mathbf a}\)-representation \(W\), we have \(\Ind^{\FI^m}(W) \subset \Ind^{\FI^m}(k[S_{\mathbf{a}}]) = M(\mathbf a)\), and therefore
\[\langle \Ind^{\FI^m}(W)_{\mathbf d}, V_{\mathbf d} \rangle_{S_{\mathbf d}} < \dim\left((V^i_{\mathbf d})^{S_{\mathbf d - \mathbf a}}\right)\]
The second condition for arbitrary \(P\) follows, since any \(P\) is a finite linear combination of the \(\chi_{\Ind^{\FI^m}(W)}\) for some irreducible representations \(W\).
\end{proof}

If a graded \(\FI^m\)-algebra \(V^*\) satisfies these two equivalent conditions, we say it is \emph{convergent}. Theorem 1.3 thus states that, under appropriate conditions, the \(\FI^m\)-algebra \(H^*_{\acute et}(\widetilde{\mathcal{Z}}^{\bullet}_n(X)_{/\overline{K}}, \mathbb{Q}_l)\) is convergent.

\begin{proof}[Proof of Theorem 1.3]
Our proof proceeds along the lines of the proofs of \cite[Thm 3.1 and Lem 7.1]{FWW}. In order to follow their proofs, we will need to recall some of the definitions they use. Recall \cite[Defn 4.1]{FWW} that the \emph{\(n\)-equals partition lattice} \(\Pi^{\mathbf d}_n\), for \(\mathbf d \in \FI^m\), is the poset of partitions of \(\mathbf d\) such that each block of the partition either has size 1, or contains at least \(n\) elements from each of the \(m\) columns. These are ordered by (reverse) refinement: \(I \le J\) if and only if \(I\) refines \(J\). We refer to blocks of size 1 as ``singleton blocks'', and the others (of size at least \(m \times n\)) as ``non-singleton blocks''. Recall that an \emph{edge} in a poset \(P\) is a pair \(a,b \in P\) with \(a < b\) and no elements between them, and a \emph{chain} of length \(r\) in \(P\) is a string \(a_0 < \cdots < a_r\) with \(a_i \in P\). Finally, recall that in \cite[Thm 4.9]{FWW}, Farb-Wolfson-Wood determine the three types of edges in \(\Pi^{\mathbf d}_n\), which we will make reference to:
\begin{description}
\item[Block creation:] A new non-singleton block with \(n\) elements each from the \(m\) columns is created from singletons.
\item[Singleton adding] A singleton block is merged with a non-singleton block.
\item[Block merging] Two non-singleton blocks are merged.
\end{description}
Now, recall that our goal is to show that
\[ \dim H^*_{\acute et}\left(\widetilde{Z}_n^{\mathbf d}(X)_{/\overline{K}}; \mathbb{Q}_l\right)^{S_{\mathbf{d} - \mathbf{a}}} \le F_{\mathbf a}(i) \]
where \(F_{\mathbf a}(i)\) is a polynomial in \(i\) (and independent of \(\mathbf d\)). Recall from the proof of \thref{etale_stab} that there is a spectral sequence \(E^{p,q} \implies H^{p+q}_{\acute et}(\widetilde{Z}_n^{\mathbf d}(X)_{/\overline{K}}; \mathbb{Q}_l)\) with
\[ E^{p,q}_2 \cong \bigoplus_{I \in \Pi^{\mathbf d}_n} H^p(X_I; \epsilon_I(q)) \]
We know that \(H^{i}_{\acute et}(\widetilde{Z}_n^{\mathbf d}(X)_{/\overline{K}}; \mathbb{Q}_l)\) is a subquotient of \(\bigoplus_{p+q=i} E^{p,q}_2\), so it is enough to prove that \(\dim (E^{p,q}_2)^{S_{\mathbf d - \mathbf a}}\) is bounded by a polynomial in \(p\) and \(q\) (uniformly in \(\mathbf d\)).

Given \(I \in \Pi^{\mathbf d}_n\), let \(S_I := S_{I_1} \times S_{I_2} \times \cdots\). Now, we have
\begin{align}
\begin{split}
(E^{p,q}_2)^{S_{\mathbf d - \mathbf a}} &= \left( \bigoplus_{I \in \Pi^{\mathbf d}_n} H^p(X_I; \epsilon_I(q))\right)^{S_{\mathbf d - \mathbf a}} = \left( \bigoplus_{I \in \Pi^{\mathbf d}_n} H^p(X_I; \epsilon_I(q))^{S_I \cap S_{\mathbf d - \mathbf a}}\right)^{S_{\mathbf d - \mathbf a}} \\
&= \left( \bigoplus_{I \in \Pi^{\mathbf d}_n} H^p\left(X_I; \epsilon_I(q)^{S_I \cap S_{\mathbf d - \mathbf a}}\right)\right)^{S_{\mathbf d - \mathbf a}}
\end{split}
\end{align}
Now, following the argument of \cite[Lem 7.1]{FWW}, we have that \(\epsilon_I(q)^{S_I \cap S_{\mathbf d - \mathbf a}} = 0\) unless \(I\) consists of exactly \(k := q/(d(mn-1)-1)\) non-singleton blocks, such that \(\#\left(I|_{\mathbf d_i - \mathbf a_i}\right) \le n\). That is, we know each of the (non-singleton) blocks \(I_j\) has at least \(n\) elements in each column, but for these invariants to be nonzero, any extras need to be in the \(\mathbf a\) coordinates. . Denote by \(\Pi' \subset \Pi^{\mathbf d}_n\) the subset of partitions satisfying this condition. It is therefore enough to take the sum in (3) over \(\Pi'\).

For a fixed \(\mathbf a\), it is clear that there are only a bounded number of ways to distribute the \(|\mathbf a|\) ``extra'' coordinates to the \(k\) blocks. This shows that \(P := \Pi'/S_{\mathbf d - \mathbf a}\), the set of all ``shapes'' of such partitions, has bounded size. We thus have
\begin{align}
\begin{split}
(E^{p,q}_2)^{S_{\mathbf d - \mathbf a}} &= \left( \bigoplus_{I \in \Pi'} H^p(X_I; \epsilon_I(q))\right)^{S_{\mathbf d - \mathbf a}} = \bigoplus_{\rho \in P} \left( \bigoplus_{I \in \rho} H^p(X_I; \epsilon_I(q))\right)^{S_{\mathbf d - \mathbf a}} \\
&= \bigoplus_{\rho \in P} \left(\bigoplus_{I \in \rho} H^p(X_I; \epsilon_I(q))\right)^{S_{\mathbf d - \mathbf a}} =  \bigoplus_{\rho \in P} \left( \Ind^{S_{\mathbf d - \mathbf a}}_{\stab I_\rho} H^p(X_{I_\rho}; \epsilon_{I_\rho}(q))\right)^{S_{\mathbf d - \mathbf a}} \\
&= \bigoplus_{\rho \in P} H^p(X_{I_\rho}; \epsilon_{I_\rho}(q))^{\stab I_\rho}
\end{split}
\end{align}
again as in \cite[Thm 3.1, p. 38-39]{FWW}, where the last equality is by Frobenius reciprocity. Here \(I_\rho\) is some partition chosen from the class \(\rho\), and the stabilizer is taken inside the group \(S_{\mathbf d - \mathbf a}\). We have
\begin{align*}
 H^p(X_{I}; \epsilon_{I}(q))^{\stab I} &= \left(H^p(X_{I}; \epsilon_{I}(q))^{S_{I} \cap S_{\mathbf d - \mathbf a}}\right)^{\stab I/(S_{I} \cap S_{\mathbf d - \mathbf a})} \\
&= H^p(X_{I}; \epsilon_{I}(q)^{S_I \cap S_{\mathbf d - \mathbf a}})^{\stab I/(S_{I} \cap S_{\mathbf d - \mathbf a})}
\end{align*}
since \(S_I\) acts trivially on \(X_I\). Next, we claim that \(\dim \epsilon_{I}(q)^{S_I \cap S_{\mathbf d - \mathbf a}}\) is bounded by a polynomial \(q\) and uniformly bounded in \(\mathbf d\). Indeed, since \(\mathbf a\) is bounded, we know that only a finite number of the non-singleton \(I_j\)'s can be larger than an \(m \times n\) block. All the rest are \(m \times n\) blocks and singletons. Thus, among all the other \(I_j\)'s of size \(m \times n\), there are only two partitions refined by \(I_j\): the complete partitions \(\hat{0}_{I_j}\), and \(I_j\) itself.

Now, the elements of \(\epsilon_I(q)\) are chains of \(\Pi^{\mathbf d}_n(\le I)\) of length \(2r(|\mathbf d| - 1) - q\). So the size of the invariants \(\epsilon_I(q)^{S_I \cap S_{\mathbf d - \mathbf a}}\) is bounded by the the number of orbits of these chains under the action of \(S_I \cap S_{\mathbf d - \mathbf a}\). But up to this group action, such chains just look like a sequence of block formations, interspersed with a \emph{bounded} number of singleton-mergers and block-mergers in the blocks larger than \(m \times n\). Notice, first of all, that this number only depends on the number of non-singleton blocks \(k = q/(d(mn-1)-1)\) and is independent of \(\mathbf d\), since all the extra singletons don't refine any nontrivial partitions. Furthermore, the number of ways to intersperse the extra moves is bounded by \(\binom{k}{a} a!\), which is a polynomial in \(q\). So our claim is proven.

Now, we have
\[ \dim  H^p(X_{I}; \epsilon_{I}(q)^{S_I \cap S_{\mathbf d - \mathbf a}})^{\stab I/(S_{I} \cap S_{\mathbf d - \mathbf a})} \le \dim H^p(X_I; \mathbb{Q})^{\stab I/(S_{I} \cap S_{\mathbf d - \mathbf a})} \cdot \dim  \epsilon_{I}(q)^{S_I \cap S_{\mathbf d - \mathbf a}} \]
Among the non-singleton blocks of any \(I \in \Pi'\), only a bounded number will have fall within the \(\mathbf a\) coordinates; denote this number by \(b\). As discussed earlier, all of the other non-singleton blocks must have size exactly \(m \times n\). Likewise there are only a bounded number of singletons in \(I\) that fall within the \(\mathbf a\) coordinates. Denote the number of singletons in the \(i\)-th column of \(I\) by \(l_i\), and the number of these in the \(\mathbf a\) coordinates by \(c_i\). Thus
\[ X_I = X^b \times X^{k-b} \times \prod_i X^{l_i - c_i} \times X^{c_i} \]
Now, \(\stab I/(S_I \cap S_{\mathbf d - \mathbf a})\) consists of those permutations that blockwise permute the \(k-b\) blocks without any \(\mathbf a\) coordinates, as well as permutations of the non-\(\mathbf a\) singletons. Thus
\[\stab I/(S_I \cap S_{\mathbf d - \mathbf a}) \cong S_{k-b} \times S_{l_1 - c_1} \times S_{l_2 - c_2} \times \dots\]
So we have
\begin{align*}
\dim H^p(X_I; \mathbb{Q})^{\stab I/(S_I \cap S_{\mathbf d - \mathbf a}} &= \dim H^p(X_I/(\stab I/(S_I \cap S_{\mathbf d - \mathbf a})); \mathbb{Q}) \\
&= \dim H^p\left( X^b \times \Sym^{k-b} X \times \prod_i X^{c_i} \times \Sym^{l_i - c_i} X; \mathbb{Q}\right)
\end{align*}
Since \(b\) and \(c_i\) are bounded, the important terms are the \(\Sym^{k-b} X\) and \(\Sym^{l_i - c_i}\). Note that, \emph{a priori}, the second seems concerning, since it depends on \(\mathbf d\), which we need our bound to be independent of. However, Macdonald \cite{Ma1} proved that \(\dim H^i(\Sym^n X; \mathbb{Q})\) is eventually independent of \(n\), and in fact has Poincar\'e polynomial given by a rational function in \(i\), with poles at roots of unity. Such a rational function is known to be bounded by a polynomial. In particular, we conclude that \(\dim H^p(X_I; \mathbb{Q})^{\stab I/(S_I \cap S_{\mathbf d - \mathbf a}}\) is bounded by a polynomial in \(p\), uniformly in \(\mathbf d\). Since we already knew this was likewise true of \(\dim  \epsilon_{I}(q)^{S_I \cap S_{\mathbf d - \mathbf a}}\), we conclude the theorem.

\end{proof}

\section{Arithmetic statistics}
Let \(Z\) be a smooth quasiprojective scheme over \(\mathbb{Z}[1/N]\). Suppose that \(S_{\mathbf d}\) acts generically freely on \(Z\) by automorphisms, and let \(Y = Z/S_{\mathbf d}\) be the quotient, which is known to be a scheme.

For any prime power \(q \nmid N\), we can base-change \(Y\) to \(\overline{\mathbb{F}}_q\). The geometric Frobenius \(\Frob_q\) then acts on \(Y_{/\overline{\mathbb{F}}_q}\). The fixed-point set of \(\Frob_q\) is exactly \(Y(\mathbb{F}_q)\).

Fix a prime \(l \nmid q\). Since all the irreducible representations of \(S_{\mathbf d}\) are defined over \(\mathbb{Q}\), \emph{a fortiori} over \(\mathbb{Q}_l\), there is a natural correspondence between finite-dimensional representations of \(S_{\mathbf d}\) over \(\mathbb{Q}_l\) and finite-dimensional constructible \(l\)-adic sheaves on \(Y\) that become trivial when pulled back to \(Z\).

Given a representation \(V\) of \(S_{\mathbf d}\), let \(\chi_V\) be its character and let \(\mathcal{V}\) the associated sheaf on \(Y\). For any point \(y \in Y(\mathbb{F}_q)\), since \(\Frob_q\) fixes \(y\), then \(\Frob_q\) acts on the fiber \(p^{-1}(y)\). Now \(S_{\mathbf d}\) acts transitively on \(p^{-1})(y)\) with some stabilizer \(H\), and so we can identify \(p^{-1}(y)\) with \(S_{\mathbf d}/H\). The \(\Frob_q\) action on \(p^{-1}(y)\) commutes with this \(S_{\mathbf d}\) action, and so it is determined by its action on a single basepoint, which we choose once and for all to be \(H\). Now \(\Frob_q(H) = \sigma_y H\) for some \(\sigma_y \in S_{\mathbf d}\). Following Gadish \cite{Ga3}, for any \(S_{\mathbf d}\)-representation \(V\) and any coset \(\sigma_H\) of \(S_{\mathbf d}\), we set
\[ \chi_V(\sigma H) = \frac{1}{|H|} \sum_{h \in H} \chi_v(\sigma h) \]
More generally, for any class function \(P\) and \(y \in Y(\mathbb{F}_q)\), we define
\[ P(y) := \frac{1}{|H|} \sum_{h \in H} P(\sigma_y h) \]
It is straightforward to show that this is independent of the choice of coset \(H\), since the action of \(S_{\mathbf d}\) is transitive on fibers. With this notation we have
\[ \tr(\Frob_q : \mathcal{V}_y) = \chi_V(\sigma_y H) \]
The Grothendieck-Lefschetz trace formula says that
\[ \sum_{y \in Y(\mathbb{F}_q)} \tr(\Frob_q : \mathcal{V}_y) = \sum_{i} (-1)^i \tr\left(\Frob_q : H^i_{\acute et, c}(Y_{/\overline{\mathbb{F}}_q}; \mathcal{V})\right) \]
We then have the following chain of equalities, as in \cite{CEF2}, \cite{FW2}, and \cite{Ca2}:
\begin{align*}
H^i_{\acute et, c}(Y_{/\overline{\mathbb{F}}_q}; \mathcal{V}) &\cong (H^i_{\acute et, c}(Z; \pi^* \mathcal{V}))^{S_{\mathbf d}} & \text{by transfer}\\
&\cong \left(H^i_{\acute et, c}(Z; L) \otimes V\right)^{S_{\mathbf d}} & \text{by triviality of pullback} \\
&\cong \left(H^{2\dim Z - i}_{\acute et}(Z; L(-\dim Z))^* \otimes V\right)^{S_{\mathbf d}} &\text{by Poincar\'e duality} \\
&\cong \langle H^{2\dim Z - i}_{\acute et}(Z; L(-\dim Z)), V \rangle_{S_{\mathbf d}}
\end{align*}
and so we obtain
\begin{equation}
\label{single_grot}
\sum_{y \in Y(\mathbb{F}_q)} \chi_V(\sigma_y H) = q^{\dim Z} \sum_i (-1)^i \tr\left( \Frob_q : \langle H^i_{\acute et}(Z; L), V \rangle_{S_{\mathbf d}}\right)
\end{equation}

We would like to apply (\ref{single_grot}) to a collection of schemes \(Z_{\mathbf d}\) that form a co-\(\FI^m\)-scheme, and then let \(\mathbf{d} \to \infty\). To make this work, we need to know that \(Z\) satisfies \'etale representation stability, and that \(H^*(Z)\) is convergent in the sense of \S3. Following \cite{FW2} and \cite{Ca2}, we have the following.

\begin{thm}
\thlabel{fig_stats}
Suppose that \(Z\) is a smooth quasiprojective co-\(\FI^m\)-scheme over \(\mathbb{Z}[1/N]\) such that \(H^i(Z_{/\overline{\mathbb{F}}_q}; \mathbb{Q}_l)\) is a finitely-generated \(\Gal(\overline{\mathbb{F}}_q/\mathbb{F}_q)\)-\(\FI^m\)-module, and that \(H^*(Z; \mathbb{Q}_l)\) is convergent. Then for any \(\FI^m\) character polynomial \(P\),
\begin{equation}
\lim_{\mathbf d \to \infty} q^{-\dim Z_{\mathbf d}} \sum_{y \in Y_{\mathbf d}(\mathbb{F}_q)} P(y) = \sum_{i=0}^\infty (-1)^i \tr\left( \Frob_q : \langle H^i_{\acute et}(Z; L), P \rangle\right)
\end{equation}
\end{thm}
\begin{proof}
Since each \(Z_{\mathbf d}\) is smooth quasiprojective, we can apply (\ref{single_grot}) to it. By linearity, we can replace a representation \(V\) of \(S_{\mathbf d}\) with a virtual representation given by a character polynomial \(P\), so we obtain
\[
q^{-\dim Z_{\mathbf d}} \sum_{y \in Y_{\mathbf d}(\mathbb{F}_q)} P(y) = \sum^{2 \dim Z_{\mathbf d}}_{i=0} (-1)^i \tr\left( \Frob_q : \langle H^i_{\acute et}(Z_{\mathbf d}; L), P \rangle_{S_{\mathbf d}}\right)
\]
Call this sum \(A_{\mathbf d}\). Furthermore, let
\[ B_{\mathbf d} = \sum_{i=0}^{2 \dim Z_{\mathbf d}} (-1)^i \tr\left( \Frob_q : \langle H^i_{\acute et}(Z; L), P\rangle\right) \]
Our goal is thus to show that
\[\lim_{\mathbf d \to \infty} A_{\mathbf d} = \lim_{\mathbf d \to \infty} B_{\mathbf d},\]
that is, first of all to show that both sides converge, and that their limits are equal.

By the assumption that \(H^*(Z)\) is convergent, we know that there is a function \(F_P(i)\) which is subexponential in \(i\), such that for all \(\mathbf d\),
\[ |\langle H^i_{\acute et}(Z_{\mathbf d}; L), P\rangle_{S_{\mathbf d}}| \le F_P(i) \]
and so by taking \(\mathbf d\) large enough,
\[ |\langle H^i_{\acute et}(Z; L), P\rangle| \le F_P(i) \]
Furthermore, by \cite[Thm 1.6]{De}, we know that
\[ \left|\tr\left( \Frob_q : \langle H^i_{\acute et}(Z_{\mathbf d}; L), P \rangle_{S_{\mathbf d}}\right)\right| \le q^{-i/2} \left|\langle H^i_{\acute et}(Z_{\mathbf d}; L), P \rangle_{S_{\mathbf d}}\right| \]
We therefore have
\begin{align*}
|A_{\mathbf d}| &\le \sum_{i=0}^{2 \dim Z_{\mathbf d}} \left|\tr\left( \Frob_q : \langle H^i_{\acute et}(Z_{\mathbf d}; L), P \rangle_{S_{\mathbf d}}\right)\right| \\
&\le \sum_{i=0}^{2 \dim Z_{\mathbf d}} \left|\langle H^i_{\acute et}(Z_{\mathbf d}; L), P \rangle_{S_{\mathbf d}}\right| \\
&\le \sum_{i=0}^{2 \dim Z_{\mathbf d}} q^{-i/2} F_P(i).
\end{align*}
By exactly the same argument, \(|B_{\mathbf d}| \le \sum_{i=0}^{2 \dim Z_{\mathbf d}} q^{-i/2} F_P(i)\). Since \(F_P(i)\) is subexponential in \(i\), this means that both \(A_{\mathbf d}\) and \(B_{\mathbf d}\) converge.

It remains to show that \(\lim_{\mathbf d \to \infty} A_{\mathbf d} - B_{\mathbf d} = 0\). Let \(N(\mathbf d, P)\) be the number such that
\[ \langle H^i_{\acute et}(Z_{\mathbf d}), P\rangle_{S_{\mathbf d}} = \langle H^i_{\acute et}(Z), P\rangle \;\; \text{ for all } i \le N(\mathbf d, P). \]
We thus have
\begin{align*}
|B_{\mathbf d} - A_{\mathbf d}| &\le \sum_{i = 0}^{2 \dim Z_{\mathbf d}} q^{-i/2} \left| \langle H^i_{\acute et}(Z), P \rangle - \langle H^i_{\acute et}(Z_{\mathbf d}), P \rangle_{S_{\mathbf d}} \right| \\
&= \sum_{i = N(\mathbf d, P) + 1}^{2 \dim Z_{\mathbf d}} q^{-i/2} \left| \langle H^i_{\acute et}(Z), P \rangle - \langle H^i_{\acute et}(Z_{\mathbf d}), P \rangle_{S_{\mathbf d}} \right| \\
&\le \sum_{i = N(\mathbf d, P) + 1}^{2 \dim Z_{\mathbf d}} 2q^{-i/2} F_P(i)
\end{align*}
Since \(N(\mathbf d, P) \to \infty\) as \(\mathbf d \to \infty\), and \(F_p(i)\) is subexponential in \(i\), we conclude that \(|B_{\mathbf d} - A_{\mathbf d}|\) becomes arbitrarily small as \(\mathbf d \to \infty\).
\end{proof}

We can now apply this to \(\widetilde{\mathcal{Z}}^{\mathbf d}_n(X)\) to obtain Theorem 1.4.

\begin{proof}[Proof of Theorem 1.4]
By Theorem 1.2, \(H^i_{\acute et}(\widetilde{\mathcal{Z}}^{\bullet}_n(X)_{/\overline{\mathbb{F}}_q}, \mathbb{Q}_l)\) is a finitely generated \(\Gal(\overline{\mathbb{F}}_q/\mathbb{F}_q)\)-\(\FI^m\)-module. By Theorem 1.3, \(H^*_{\acute et}(\widetilde{\mathcal{Z}}^{\bullet}_n(X)_{/\overline{\mathbb{F}}_q}, \mathbb{Q}_l)\) is convergent. We thus conclude by Theorem 5.1.
\end{proof}

\bibliography{general}

\end{document}